\documentclass[12pt,psamsfonts]{article}
\usepackage{amsmath}
\usepackage{amssymb}
\usepackage{amsthm}

\newtheorem{defn0}{Definition}[section]
\newtheorem{prop0}[defn0]{Proposition}
\newtheorem{thm0}[defn0]{Theorem}
\newtheorem{lemma0}[defn0]{Lemma}
\newtheorem{corollary0}[defn0]{Corollary}
\newtheorem{example0}[defn0]{Example}
\newtheorem{remark0}[defn0]{Remark}
\newtheorem{conjecture0}[defn0]{Conjecture}
\newtheorem{claim0}[defn0]{Claim}

\newenvironment{definition}{ \begin{defn0}}{\end{defn0}}
\newenvironment{proposition}{\bigskip \begin{prop0}}{\end{prop0}}
\newenvironment{theorem}{\bigskip \begin{thm0}}{\end{thm0}}
\newenvironment{lemma}{\bigskip \begin{lemma0}}{\end{lemma0}}
\newenvironment{corollary}{\bigskip \begin{corollary0}}{\end{corollary0}}

\newenvironment{remark}{ \begin{remark0}\rm}{\end{remark0}}

\newcommand{\propref}[1]{Proposition~\ref{#1}}
\newcommand{\thmref}[1]{Theorem~\ref{#1}}

\newcommand{\corref}[1]{Corollary~\ref{#1}}

\def\cocoa
{\mbox{\rm C\kern-.13em o\kern-.07
em C\kern-.13em o\kern-.15em A}}

\newcommand{\m}{\mathfrak{m}}
\newcommand{\n}{\mathfrak{n}}

\newcommand{\emdim}{{\rm emdim}}
\newcommand{\depth}{{\rm depth}}
\newcommand{\socle}{{\rm socle}}

\newcommand{\spec}{{\rm Spec}}

\def\Hilb{{{\mathcal H}\kern -0.25ex{\it ilb\/}}}

\title{\bf \huge Poincar\'{e} series and deformations of  Gorenstein local algebras  with low socle degree
}

\author{\large   G. Casnati
\and \large J. Elias
\thanks{Partially supported by  MTM2007-67493, Acci\'{o}n Integrada Espa\~{n}a-Italia 07-09}
\and \large R. Notari
\and \large M. E. Rossi
\thanks{Partially supported by M.I.U.R.: PRIN 07-09, Azione Integrata Italia-Spagna 07-09 }
}

\date{\today}

\begin{document}

\maketitle

\begin{abstract}
Let $K$ be an algebraically closed field of characteristic $0$, and let $A $ be an Artinian  Gorenstein local commutative and Noetherian  $K$--algebra, with  maximal ideal $\m$.
 In the present paper we prove a structure theorem describing such kind of $K$--algebras satisfying $\m^4=0$. We use this result   in order to prove that  such a $K$--algebra $A$ has rational  Poincar{\'e} series and it is always smoothable in any embedding dimension, if $\dim_K \m^2/\m^3 \le 4$.
 We also prove that the generic Artinian Gorenstein local $K$--algebra with socle degree three has rational Poincar{\'e} series, in spite of the fact that such algebras are  not necessarily smoothable.
\end{abstract}

\bigskip
\section{Introduction}

Throughout the whole paper  by $K$--algebra we will mean a local commutative Noetherian  $K$--algebra. When $A$ is a local ring we will usually denote by $\m$ its unique maximal ideal. We always assume that $K$ is algebraically closed of characteristic zero.
The socle degree of an Artinian algebra $A$ is the integer $s$ such that $\m^s \neq 0,$ but $\m^{s+1}=0$.

In  the main result of this paper we present a structure theorem for the  Artinian Gorenstein $K$--algebras $A$ of socle degree $3$, \thmref{structure}. The key ingredients of the proof  are some well--known facts about the classical Macaulay's correspondence and a recent result of J. Elias and M.E. Rossi (see  \cite{ER}).
From this result we get some applications to the rationality of the Poincar{\'e} series
$$
{ {P}}_A(z):=\sum_{i\ge 0}\dim_KTor_i^A(K, K)z^i
$$
and the smoothability of $\spec(A)$ in the Hilbert scheme $\Hilb_d({\mathbb A}^h)$
parameterizing punctual subschemes of multiplicity  $d$ of the affine space of dimension $h$ over $K$.

Actually the  results on the Poincar\'{e} series of an Artinian  Gorenstein local ring could  be stated in any dimension $r.$ In fact, let $(A, \m) $ be a Gorenstein local ring of dimension $r $ and let  $J$ be an ideal generated by a minimal reduction of $\m$, see \cite{BH93}.  Then by \cite{Sch}, Satz 1,
$$P_A(z) =(1+z)^r P_{A/J}(z) $$
where $A/J $ is Artinian, still  Gorenstein,  whose length is the multiplicity of $A.$

Notice that, when $\m^r=0$ for $r=1,2$, it is completely trivial to prove the rationality of $P_A(z)$ and the smoothability in $\Hilb_d({\mathbb A}^h)$ since, in those cases, $A$ has multiplicity either $1$ or $2$. When $\m^3=0$ both such properties again hold  true for Artinian Gorenstein $K$--algebras of arbitrary multiplicity $d$ (see \propref{cube}).

J.P. Serre in \cite{serre} conjectured the rationality of ${ {P}}_A(z) $ and proved it when $A$ is a regular local ring. The problem has   motivations in commutative and non-commutative algebra and over the last 30 years has been  a topic of much current interest. Despite many interesting results showing the rationality, J. Anick (see \cite{A}) has given an example of an Artinian algebra $A$  such that $\m^3=0$ with transcendental    Poincar{\'e} series.     G. Sj\"odin proved that all Artinian Gorenstein   rings with $\m^3 =0$  have a rational Poincar{\'e} series
(see \cite{Soj79}), R. B{\o}gvad proved that there exist   Artinian Gorenstein  local rings with $\m^4=0 $   and  transcendental Poincar{\'e} series (see \cite{Bog83}).
Nevertheless,  several results show that large classes of Gorenstein local rings  have a good behaviour with respect to this problem. We mention the local rings which are complete intersections (see \cite{tate}), Gorenstein rings with $\emdim(A) -4 \le  \depth (A) $ (see \cite{AKM} and \cite{JKM}), stretched and almost stretched Gorenstein rings (see \cite{sally2}, \cite{EV1, EV2}), i.e. such that $\dim_K  \m^2/\m^3 \le 2,  $   Gorenstein local rings such that $\dim_K \m^2/\m^3 =3 $ and $\m^4=0$ (see \cite{CN3}).

In Section 3  we deal with the case of an Artinian Gorenstein local $K$--algebra $A$ such that $\m^4=0  $    giving  necessary and sufficient conditions for being ${ {P}}_A(z)  $  rational.  In particular we prove  that $A$ has rational Poincar{\'e} series if  $\dim_K \m^2/\m^3 \le 4 $  (see Corollary \ref{n4}).

In \cite{Mazzola1}, \cite{Mazzola2} and  \cite{CEVV} it is proved there that $\Hilb_d(\mathbb A^h)$ is irreducible if either $d\le7$ or $d=8$ and $h\le3$ and it has exactly two components when $d=8$ and $h\ge4$. In particular each Artinian $K$--algebra $A$ of length $d\le7$ is smoothable i.e. can be flatly deformed to the trivial $K$--algebra $K^d$.
In \cite{CN2} and \cite{CN4} the attention is focused on the problem of smoothability of Gorenstein algebras in $\Hilb_d(\mathbb A^h)$.
The locus $\Hilb_d^G(\mathbb A^h)\subseteq\Hilb_d(\mathbb A^h)$ of such $K$--algebras contains all the algebras isomorphic to the trivial $K$--algebra $K^d$, thus it makes sense to ask when $\Hilb_d^G(\mathbb A^h)$ is irreducible and what kind of algebras  it contains. In the quoted papers it is proved that this is actually true if $d\le10$ whereas in \cite{mattone} the authors prove the reducibility of  $\Hilb_d^G(\mathbb A^h)$ when $d\ge14$ asserting the existence of a non smoothable Artinian Gorenstein local $K$--algebra with Hilbert function $(1, 6,6, 1)$ (see also \cite{CN4} Section 4 for an explicit example).

In Section 4, starting from the structure theorem for Artinian Gorenstein $K$--algebras $A$ with $\m^4=0$, proved in Section 2,  we achieve some results concerning the smoothability of $\spec(A)$ in the Hilbert scheme $\Hilb_d({\mathbb A}^h)$.
This problem can be reduced to the Artinian Gorenstein graded  algebras with Hilbert function $(1,h, h, 1)$ which have been well treated in the literature.
In Corollary \ref{n8} we prove  that $A$ is smoothable if  $\dim_K \m^2/\m^3 \le 4 $.

\bigskip
\section{Structure theorem of Artin Gorenstein algebras with socle degree three}

An Artinian  Gorenstein  algebra $A$ with maximal ideal $\m$ over a field $K$ is self dual, that is there exists an exact pairing from $A \times A $ to $K$ making $A$ isomorphic as $A$--module to $Hom_K(A,K).$

The Hilbert function of $A$ is by definition the Hilbert function of the associated graded algebra $G= gr_{\m}(A) := \bigoplus_{i \ge 0} \m^i/\m^{i+1}$ (by definition $\m^0=A$), i.e.
$$
HF_A(i) =   \dim_K \ \m^i /\m^{i+1}.
$$
In the graded case,   the Hilbert function of an Artinian Gorenstein algebra is symmetric.
Little is  known about the Hilbert function in the local case. The problem comes from the fact that  the associated  graded algebra  $ G$ is in general no longer  Gorenstein.

Nevertheless  A. Iarrobino in \cite{IAMS} proved interesting results  even in the local case.  The information comes from a stratification of $G$ by a descending sequence of ideals
$$G=C(0) \supset C(1) \supset \dots, $$ whose successive quotient $Q(a) =C(a)/C(a+1) $ are reflexive $G$-modules. This reflexivity property imposes conditions on the Hilbert function $HF_A. $ If $A$ has socle degree $s$, then $HF_A$ is a sum of symmetric functions $ HF_{Q(a)}$ with respect  $ \frac{s-a}{2}$.

The first subquotient $Q(0) $ of $G$  is always a graded Gorenstein algebra and it is the unique (up to isomorphism) quotient of $G$ with socle degree $s.$   If necessary, we will write $Q_A(0).$
A. Iarrobino proved that if $HF_A$ is symmetric, then $G= Q_A(0) $ and it is Gorenstein.
 Hence (see \cite{IAMS} Proposition 1.7 and \cite{E} Proposition 7)
{ { $G$  is Gorenstein  if  and only if  $HF_A$ is symmetric, equivalently if   $G=Q_A (0).  $} }

For example, if $A$ is a Gorenstein local $K$--algebra with socle degree  three and embedding dimension $h $ ($=HF_A(1)$),  then we deduce that $HF_A(2)=n \le h  $  and clearly $HF_A(3)=1.$ In this case  we will write that $A$ has Hilbert function $ (1, h, n, 1)$; notice that  $Q_A(0) $ has Hilbert function  $(1, n, n, 1).$

 Macaulay's inverse system has an important role in the  study   of  Artinian  local $K$--algebras.
The reader should refer to \cite{E},  \cite{IAMS} and \cite{ER} for an extended treatment. If $R$ is a local ring over a field $K, $ we may regard the dual module $Hom_K(A,K) $ of a quotient $A=R/I $ as a submodule of the injective envelope $Hom_K(R,K). $

If $R= K[\![x_1,\dots x_h]\!]$ is a power series ring   with maximal ideal $\n =(x_1,\cdots,x_h), $ then, being $K$  a characteristic zero field,
$Hom_K(R,K) $ is a divided power ring $S=K[y_1, \dots, y_h] $ and we get  an explicit description of the duality.

Since $A$, being Artinian, is complete with respect to the $\m$--adic topology, hence we may assume $A$ is a quotient of $R. $   \vskip 2mm

  It is known that
$S $ has a structure of   $R $--module by means the following action
$$
\begin{array}{ cccc}
\circ: & R \times S   &\longrightarrow &  S   \\
                       &       (f , g) & \to  &  f  \circ g = f (\partial_{y_1}, \dots, \partial_{y_h})(g)
\end{array}
$$
where $  \partial_{y_i} $ denotes the partial derivative with respect to $y_i.$

J. Emsalem in \cite{E}, Section B, Proposition 2, A. Iarrobino in   \cite{IAMS} Lemma 1.2,
 characterized Artinian local $K$--algebras in terms of suitable $R$--submodules  of $S $ which are finitely generated.

A  quotient local ring $A=R/I $ is an Artinian   Gorenstein  local $K$--algebra of socle degree $s $ if and only if its dual module   is a cyclic $R$--submodule of $S$ generated by a polynomial  $F \in S $ of degree $s $ (see also  \cite{IAMS}, Lemma 1.2.).   If we consider the ideal of $R$
$$
Ann_{R }(F): =\{g \in R  \ |\  g \circ F = 0 \   \}
$$
then $A= R/Ann_{R}(F).$
 \noindent Hence each  Artinian Gorenstein local $K$--algebra of socle $s$ will be equipped with a polynomial $F \in S$ of degree $s.$  We  will write $A=A_F.$ The polynomial $F$ is not unique, but it  is determined up an unit $u$ of $R$.   In the general case, one can translate  in terms of classification and deformation of polynomials, the respective problems of classification and deformation of the corresponding  algebra.
For example, when $A$ is an Artinian Gorenstein  local $K$--algebra of socle  $s=2$ and embedding dimension $h$, we have $A \simeq A_F$ with $F =y_{1}^2+ \dots + y_h^2 \in S$ due to the classification of quadrics up to projectivities.  We have thus the following result.

\begin{proposition}
\label{cube}
An Artinian Gorenstein local $K$--algebra $A$ of embedding dimension $h > 1$ has  socle  $s=2$  if and only if $A\cong R/I$
where
$$
 I=(x_i x_j, x_u^2-x_1^2)_{1\le i<j\le h , \; u=2,\dots, h}.
$$
In this case  $P_A(z) $ is rational and $A$ is smoothable.
\end{proposition}
\begin{proof}
The first part is an immediate consequence of the classification of quadrics up to projectivities (see also \cite{sally2}). For the rationality of $P_A(z)$ see e.g. \cite{Soj79} or \cite{CRV}, Proposition 2.12. For the smoothability see \cite{CN1}, Section 3.
\end{proof}

It is interesting to point out that the  Gorenstein assumption on $A$  is necessary because Anick's example has socle degree $s=2,$   nevertheless $P_A(z) $ is not rational.
It is thus a natural question to ask what happens when $s>2. $

\medskip
\noindent The $G$--module $Q(0) $ will play  a crucial role in this investigation. It  can be computed in terms of the corresponding polynomial in the inverse system. Let $F \in S  $ be a polynomial of degree  $ s $ such that $A=A_F $ and denote by $F_s$ the form of highest degree in $F, $ that is  $F =F_s + \text{terms of lower degree}$,  then  (see  \cite{E} Proposition 7 and \cite{IAMS} Lemma 1.10)
$$
Q(0) \simeq R/Ann_{R}(F_s).
$$

 \vskip 2mm
\noindent We say that a homogenous form $F\in S$ of degree $d$
is {\it{non-degenerate}}  if the $K$--vector space of the derivatives of order $d-1$ has maximal dimension, that is  $h= \dim_K [S]_1$.

Now we focus our attention on the case $s=3$, i.e. the case of Artinian  Gorenstein local algebras $A$ with Hilbert function $(1, h, n, 1)$ with $h \ge n$. Notice that  the Hilbert function of $Q(0)\simeq R/Ann_{R}(F_3) $ is $(1, n, n, 1)$.

\vskip 2mm
From now on we let  $R_j=K[[x_1, \dots, x_j]] $ and $S_j=K[ y_1, \dots, y_j]   $ for every positive integer $j \le h.  $  Hence $R_h=R $ and $S_h=S $ and in this case we will write $R$ and $S.$  In the following  $h, n $ will denote positive integers such that $n \le h.$ We assume $A=R/I$ of embedding dimension $h, $ that is $I \subseteq  \n^2.$

\vskip 2mm
J. Elias and M.E. Rossi proved the following result.
\begin{theorem} [\cite{ER}, Theorem 4.1]
\label{class}
Let $A$ be an Artinian  Gorenstein local $K$--algebra with Hilbert function $(1, h ,n,1), $ then $A\cong R/Ann_{R}(F)$
where
$$
 F =F_3 +y_{n+1}^2+ \dots + y_h^2 \in S
 $$
 with $F_3 $ a  non-degenerate degree three  form  in $S_n$ ($F=F_3$ if $n=h$).
\end{theorem}

\bigskip
Starting from the above result we can prove a structure theorem for Artinian  Gorenstein local $K$--algebra $A=R/I$ with maximal ideal $\m$ and socle degree three.

\begin{lemma} \label{ideal} Let $h, n$ be positive integers such that $h>n $ and let $$ F=  F_3 +y_{n+1}^2+ \dots + y_h^2 \in S $$  where  $F_3 $ a  non-degenerate degree three  form  in $S_n.$
Then
$$ Ann_R(F)=  Ann_{R_n}(F_3)R  + (x_i x_j, x_j^2-2 \sigma(x_1,\cdots, x_n))_{i<j, \;n+1\le j  \le h} $$
where $ \sigma\in R_n$ is any  form of degree $3$ such that $ \sigma  \circ F_3 =1.$
\end{lemma}
\begin{proof}   It is easy to check that
$$
I:= Ann_{R} (F) \supseteq J:= Ann_{R_n}(F_3)R   + (x_i x_j, x_j^2- 2 \sigma(x_1,\cdots, x_n))_{i<j, \;n+1\le j \le h}
$$
  Since ${R}/I $ and ${R}/J$ are finitely generated $K$-vector spaces ($\n^4 \subseteq I, J $) and there is a surjection between ${R }/J $ and ${R }/I, $ the equality $I=J$ follows if they have the same colength.  In particular we prove  that $HF_{{R }/I}(i) =HF_{{R }/J}(i) $ for every $i \ge 0.$
If we denote by $(\ )^* $ the homogeneous ideal  generated by the initial forms of an ideal of $R,$  we have
$$
L= (Ann_{R_n}(F_3))^*  +(x_i x_j, x_j^2)_{i<j, \;n+1\le j  \le h} \subseteq J^*  \subseteq I^*
$$
thus
$$
HF_{{R}/L}(i) \ge HF_{R/J} (i) \ge HF_{R/I} (i)
$$
for every $i \ge 0$. In order to prove the assertion we prove  that the Hilbert function of $R/L   $  is  $(1, h, n,1), $ i.e. the Hilbert function of $R/I $ (see  \thmref{class}). The equality easily follows since
   ${R_n}/Ann_{R_n}(F_3) $ has Hilbert function $(1,n,n,1)$, being $F_3$ a  non-degenerate degree three  form  in $S_n $ and $L_2= [Ann_R(F_3)^*]_2.$
\end{proof}

\begin{theorem} \label{structure}
Let $A$ be an Artinian   local $K$--algebra  of embedding dimension $h $ and let $n=HF_A(2).$  \par \vskip 2mm
\noindent $A$ is  Gorenstein of  socle degree three  if and only if  $n \le h$ and there exists   a  non-degenerate cubic form  $ F_3\in S_n$   such that $A\cong R/I$
where
$$
I=\left\{
\begin{array}{ll}
 Ann_{R_n}(F_3)R  + (x_i x_j, x_j^2-2 \sigma(x_1,\cdots, x_n))_{i<j, \;n+1\le j  \le h}  & \text{ if  }   n<h \\ \\
Ann_{R}(F_3) &\   \text{if }    n=h
\end{array}
\right.
$$
 being  $ \sigma\in R_n$ any form of degree $3$ such that $ \sigma  \circ F_3 =1. $
\end{theorem}
\begin{proof}
Let $A$ be  an  Artinian Gorenstein local $K$--algebra $A$ with  socle degree three.  The Hilbert function of $A$ is of the form $(1, h ,n,1). $  Due to the decomposition theorem proved in \cite{IAMS}, then $h\ge n$, hence \thmref{class} yields that $A\cong R/Ann_{R}(F)$
where
$$
 F =F_3 +y_{n+1}^2+ \dots + y_h^2 \in S
$$ if $h>n$ and $F=F_3$ if $h=n. $
Then the result follows now by Lemma \ref{ideal}.

\noindent Conversely, the result follows by Macaulay's correspondence and again by  Lemma \ref{ideal}.
\end{proof}

\bigskip
\section{Poincar{\'e} series of Gorenstein local algebras with socle degree three}

In this section we will reduce the computation of the Poincar{\'e} series of Gorenstein local $K$--algebras with socle degree three to the Poincar{\'e} series  of  graded Gorenstein algebras with the same socle degree.
By \cite{Sch}, Satz 1, we may assume that $A$ is an    Artinian Gorenstein algebra.

\vskip 2mm
\noindent We recall that
 if $x\in \m\setminus \m^2$ is  an element in the socle $(0:_A\m)$ of $A$,
then, by  \cite{GL},
\begin{equation}
\label{elsocle}
{ {P}}_A(z)=\frac{{ {P}}_{A/xA}(z)}{1-z \ {  {P}}_{A/xA}(z)}.
\end{equation}

\noindent Moreover if $A$ is an  Artinian   Gorenstein local ring, then, by  \cite{AL},
\begin{equation}
\label{socle}
{ {P}}_A(z)=\frac{{ {P}}_{A/(0:\m)}(z)}{1+z^2 \ { {P}}_{A/(0:\m)}(z)}
\end{equation}

\vskip 2mm \noindent
If  $A=A_F $ with $F=F_3+ \dots $ we will write   the Poincar{\'e} series of $A$ in terms of those of the Artinian Gorenstein graded  $K$--algebra $Q_A(0) \simeq {R}/Ann_{R}(F_3).  $   Notice that if $A$ has  Hilbert function $(1, h, n, 1)$, then $n \le h$ and $Q_A(0) $ is a Gorenstein graded algebra with  Hilbert function $(1, n, n, 1) $.

\begin{theorem} \label{main}  Let $A$ be an Artinian Gorenstein local $K$--algebra with Hilbert function $(1, h, n, 1). $
 Then
 $$
 {P}_A(z)=  \frac{ {P}_{Q_A(0)}(z)}{1 - (h-n) \;  P_{Q_A(0)}(z)}.
$$
In particular $ {P}_A(z)$ is rational if and only if $ {P}_{Q_A(0)}(z)$
 is rational.
\end{theorem}
\begin{proof} We recall that   $n \le h. $    By  \thmref{class},  we may assume $n < h$ and $A = {R}/Ann_{R}(F) $ where, as usual, ${R }=K[\![x_1,\dots x_h]\!] $
and $F =F_3 +y_{n+1}^2+ \dots + y_h^2,
 $
 with $F_3 $ a  non-degenerate degree three  form  in $ S_n$.
By  Lemma \ref{ideal}, we know that
$$
Ann_{R}(F)=Ann_{R_n}(F_3)R + (x_i x_j, x_j^2-2  \sigma(x_1,\cdots, x_n))_{i<j, \;n+1\le j  \le h}
$$
where $ \sigma\in S_n$ is a cubic form such that $\sigma \circ F_3 =1.$
So the coset of $ \sigma$ in $A$ is a generator of the socle
$(0:_A \m)$ of $A$, $\m$ being the maximal ideal of $A$.
Hence by  (\ref{socle}) we get that
$$
 {P}_A(z)= \frac{ {P}_C (z)}{1 + z^2 \; {P}_C(z)}
$$
where
$$
C:=\frac{A}{(0:_A \m)}\cong \frac{{R} }{Ann_{R_n}(F_3) R+ ( \sigma)+ (x_i x_j, x_j^2)_{i<j, \;n+1\le j \le h}}.
$$
Since $x_h \in \socle(C)$, by (\ref{elsocle}), we get that
$$
 {P}_C(z)=\frac{ {P}_{C/(x_h)}(z)}{1 -  z \; {P}_{C/(x_h)}(z)}.
$$
Iterating the process we deduce that
$$
 {P}_C(z)=\frac{{ P}_{D}(z)}{1 -  (h-n) z \;  {P}_{D}(z)}.
$$
where
$$
D=\frac{R }{Ann_{R_n}(F_3) R+ ( \sigma) + (x_{n+1},\cdots , x_h)}.
$$
Since $F_3,  \sigma \in S_n$ we get that
$$
D\cong \frac{R_n}{Ann_{R_n}(F_3)+ ( \sigma)}.
$$
Notice that again the coset of $ \sigma$ in $B:=R_n/Ann_{R_n}(F_3) \simeq R/Ann_{R}(F_3) =Q_A(0) $ is a generator of its socle,
so
$$
D\cong \frac{R_n}{Ann_{R_n}(F_3)+ ( \sigma)}\cong \frac{B}{\socle(B)}.
$$
Hence from (\ref{socle}) we deduce
$$
 {P}_D(z)=\frac{ {P}_B}{1 - z^2 \; {P}_B}.
$$
From the above information, summing up,  we get
$$
 {P}_A(z)=\frac{ {P}_B}{1 - (h-n) \;  {P}_B}.
$$
Since $ {P}_A$ is a rational function of $P_B$ it follows that $ {P}_A$ is a rational if and only if the same is true for $P_B$.\end{proof}

\bigskip
The above result reduces the problem of the rationality of $P_A(z) $ to the rationality of the Poincar{\'e} series of a graded Gorenstein $K$--algebra with socle degree three. This situation has been studied by a great number of researchers.    By taking advantage of this we present the following corollaries.

\begin{corollary}
\label{n4}
Let   $A$ be an Artinian Gorenstein local $K$--algebra with Hilbert function $(1, h, n, 1)$.
If $n \le 4, $ then $ {P}_A(z)$ is rational.
\end{corollary}
\begin{proof} By \thmref{main} we may reduce the problem to graded Gorenstein $K$--algebras  with embedding dimension $n \le 4.$ Hence the result follows by  \cite{AKM} and \cite{JKM}. \end{proof}

\vskip 2mm
We recall that a  Koszul graded algebra $C$ has rational Poincar{\'e} series, in particular $P_C(z) H_C(-z) =1 $ where $H_C(z) $ is the Hilbert series of $G, $ i.e.  $  H_C(z)= \sum_{j\ge 0} HF_C(j) z^j.$  The following result specializes,   for Gorenstein local $K$--algebras with socle degree three, a result of Fr\"{o}berg in \cite{froberg}.

\begin{corollary}
\label{n1}
Let $A$ be an Artinian Gorenstein local $K$--algebra   with Hilbert function $(1, h, n, 1)$.
If $Q_A(0) $ is Koszul,  then $$ {P}_A(z)= \frac{1}{n-h+1 -nz+nz^2-z^3} $$
\end{corollary}
\begin{proof}  By \thmref{main} we may reduce the problem to graded Gorenstein $K$--algebras  which are Koszul. Since $Q_A(0) $ has   Hilbert series $H(z) = 1 +nz+nz^2+z^3, $ then
$ P_{Q_A(0)} H(z)=1 $ and the result follows from an easy computation.
\end{proof}

\noindent In  \cite{CTV}  Koszul filtrations  were  introduced for studying the Koszulness
of a standard graded algebra $C.$   It has been proved in \cite{CTV}, Proposition 1.2 that if $C$ has a  Koszul filtration then all the ideals of the filtration have  a linear $C$--free resolution, hence $C$ is Koszul.
For
Gorenstein graded algebras with Hilbert function $(1, n, n, 1) $, the property of  having a Koszul filtration   can be detected directly on the cubic form (see \cite{CRV} and \cite{conca} Theorem 3.2).

Notice that if  $F_3 $ is a generic cubic, then the corresponding Gorenstein algebra $R /Ann_{R }(F_3)=Q_A(0)  $ has a Koszul filtration (see \cite{CRV}, Theorem 6.3), hence
 $ {P}_A(z)$ is rational.

We say that $A_F$ is a generic  Artinian Gorenstein local $K$--algebra if   $F$ is a generic polynomial of $S. $ By the previous remark we get the following result.

\begin{corollary} \label{generic}
The generic Artinian Gorenstein local $K$--algebra with Hilbert function $(1, h, n, 1)$  has rational Poincar{\'e} series.
\end{corollary}

\medskip
\begin{remark}
J. Elias and G. Valla proved that if the multiplicity of $A$ is either $d \le 9$ or $d \le h+4, $ then
$ P_A(z) $ is rational (see \cite{EV3}).
If $d\le h+6$ then by \corref{n4} and \cite{EV3} we get that $ P_A(z) $ is rational but the case that
$A$ has Hilbert function $ (1,h,3,1,1)$ is still open.
\end{remark}

\bigskip
\section{Deformations of Gorenstein local algebras with socle degree three}

We start this section by recalling some facts about families of $K$--algebras of fixed degree. Let $d\ge2$ be an integer and let $A$ be an Artinian $K$--algebra of length $d$, hence $A\cong K^d$ as $K$--vector spaces. As explained in the introduction it is interesting to understand weather such an $A$ is smoothable in the following sense.

\begin{definition}
An Artinian $K$--algebra $A$ of length $d$ is smoothable if the scheme $\spec(A)\in \Hilb_d(\mathbb A^d)$ is in the closure $\Hilb_d^{gen}(\mathbb A^d)$ inside $\Hilb_d(\mathbb A^d)$ of the locus of points representing reduced schemes.
\end{definition}

We focus here our attention to Artinian Gorenstein local $K$--algebras $A$ with Hilbert function $(1,h,n,1)$.  In this case we are able to say something interesting using Lemma \ref{ideal}.

\begin{theorem}
\label{smooth}
Let $A$ be an Artinian Gorenstein local $K$--algebra with Hilbert function $(1, h, n, 1). $
Then $A$ is smoothable if $ Q_A(0)$ is smoothable.
\end{theorem}
\begin{proof}
We will extend the methods used in \cite{CN1} and \cite{CN2}, proving the statement by induction on $h$.

We will assume the statement true for $h-1$ and we will prove it for $h$. Recall that $A\cong {R}/Ann_{{R}}(F)$ where (see  \thmref{class})
$$
 F =F_3 +y_{n+1}^2+ \dots + y_h^2 \in S=K[y_1, \dots, y_h]
 $$
with $F_3 $ a  non-degenerate degree three  form  in $ K[y_1, \dots ,y_h]$. If $h=n$, then
the ring $A\cong {R}/Ann_{{R}}(F_3)$ is graded Artinian Gorenstein and local with socle in degree $3$ and Hilbert function $(1,n,n,1)$. Thus $A\cong G\cong Q_A(0)$ and the statement is trivial in this case.

We assume the statement is true for $h-1$: we will prove it for $h>n$. Recall that we have
(see \thmref{main})
$$
Ann_{R}(F)=Ann_{R_n}(F_3)R+ (x_i x_j, x_j^2- 2  \sigma(x_1,\cdots, x_n))_{i<j, \;n+1\le j  \le h}
$$
where $s$ is a degree three  form in $ S_n$ such that $ \sigma \circ F_3 =1.$  For the techniques used in this proof,  it will be useful working in the polynomial ring. Notice that in our setting {\bf we may assume}   $R=K[x_1, \dots, x_h] $ because in $R/Ann_{R}(F)$ we have $\m^4=0.$

For each $b\in {\mathbb A}^1$, let us  consider the ideal in $R$
\begin{multline*}
J_b:=Ann_{R_n}(F_3)R +\\(x_ix_j,x_k^2-2 \sigma(x_1,\dots,x_n),x_h^2-bx_h-2 \sigma(x_1,\dots,x_n))_{i<j, \;n+1\le j  \le h, n+1 \le k < h}.
\end{multline*}
Notice that
$$
J_0=Ann_{R}(F)=Ann_{R_n}(F_3)R+ (x_i x_j, x_j^2- 2  \sigma(x_1,\cdots, x_n))_{i<j, \;n+1\le j  \le h},
$$
hence $A\cong R/J_0$.

If $b\ne0$, we claim that
$$
J_b=(x_1,\dots,x_{h-1},x_h-b)\cap(J_b+(x_h^2))\subseteq  R.
$$
Since  $J_b\subseteq (x_1,\dots,x_{h-1},x_h-b),  $  by the modular law (see \cite{AM}, Chapter I), it is enough to prove
$$
(x_1,\dots,x_{h-1},x_h-b)\cap( x_h^2 )=(x_1x_h^2,\dots,x_{h-1}x_h^2,(x_h-b)x_h^2)\subseteq J_b.
$$
This is trivial because $x_i x_h \in J_b $ for every $i=1, \dots {h-1} $ and $x_h^2(x_h-b) \equiv 2 x_h  \sigma(x_1,\dots,x_n)) \equiv 0 $ modulo $J_b.$

If $b\ne0$, then $x_h\in J_b+(x_h^2)$, thus  the ideals $(x_1,\dots,x_{h-1},x_h-b)$ and $J_b+(x_h^2)$ are coprime, whence
$$
A_b:=R/J_b\cong R/(x_1,\dots,x_{h-1},x_h-b)\oplus R/(J_b+(x_h^2)).
$$
Since $x_h\in J_b+(x_h^2)$ we also have
$$
R/(J_b+(x_h^2))\cong A':=R_{h-1}/(Ann_{R_n}(F_3)R+ (x_i x_j, x_j^2- 2  \sigma(x_1,\cdots, x_n))_{i<j, \;n+1\le j  \le h-1}).
$$
By induction hypothesis $A'$ is smoothable. It follows easily that $A_b\cong K\oplus A'$ turns out to be smoothable for $b\ne0$ by induction hypothesis (for reader's benefit see e.g. Lemma 4.2 of \cite{CEVV}).

Let ${\mathcal R}:= K[b]\otimes_K R\cong K[b,x_1,\dots,x_h]$ and  consider the family ${\mathcal A}\cong {\mathcal R}/J\to {\mathbb A}^1\cong\spec(k[b])$ where
\begin{multline*}
J:=Ann_{R_n}(F_3){\mathcal R} +\\ (x_ix_j,x_k^2-2 \sigma(x_1,\dots,x_n),x_h^2-bx_h-2 \sigma(x_1,\dots,x_n))_{i<j, \;n+1\le j  \le h, n+1 \le k < h}.
\end{multline*}
Due to the discussion above all the fibres of ${\mathcal A}\to {\mathbb A}^1$ are Artinian $K$--algebras of degree $d$, thus the family is flat. The universal property of Hilbert scheme guarantees the existence of a curve inside  $\Hilb_d(\mathbb A^d)$ whose general point is in  $\Hilb_d^{gen}(\mathbb A^d)$, thus $\spec(A)\in\Hilb_d^{gen}(\mathbb A^d)$.
\end{proof}

\bigskip
As already noticed above \thmref{smooth} reduces the smoothability of $A$ to the smoothability of a graded Artinian Gorenstein local $K$--algebra with socle degree three (see e.g. \cite{IE}). Our lastresult is the following corollary

\begin{corollary}
\label{n8}
Let   $A$ be an Artinian Gorenstein local $K$--algebra with Hilbert function $(1, h, n, 1)$.
If $n \le 4, $ then $A$ is smoothable.
\end{corollary}
\begin{proof} It follows by Theorem \ref{smooth} and \cite{CN4}.
\end{proof}

We remark that the previous result cannot be generalized for $n \ge 6.$  Indeed Iarrobino  found an example of non smoothable local $K$--algebra $A$ with Hilbert function $(1,6,6,1)$
(see  \cite{CN4}, Section 4).
The case $n=5$ is still open.

\bigskip
\providecommand{\bysame}{\leavevmode\hbox to3em{\hrulefill}\thinspace}

\bigskip

\bigskip
\noindent
Gianfranco Casnati\\
Dipartimento di Matematica\\
Politecnico di Torino\\
Corso Duca degli Abruzzi 24, 10129 Torino, Italy\\
e-mail: {\tt gianfranco.casnati@polito.it}

\bigskip
\noindent
Juan Elias\\
Departament D'{\`A}lgebra i Geometria \\
Facultat de Matem\`{a}tiques\\
Universitat de Barcelona\\
Gran Via 585, 08007 Barcelona, Spain\\
e-mail: {\tt elias@ub.edu}

\bigskip
\noindent
Roberto Notari\\
Dipartimento di Matematica\\
Politecnico di Milano\\
Via Bonardi 9, 20133 Milano, Italy\\
e-mail: {\tt roberto.notari@polimi.it}

\bigskip
\noindent
Maria Evelina Rossi\\
Dipartimento di Matematica\\
Universit{\`a} di Genova\\
Via Dodecaneso 35, 16146 Genova, Italy\\
e-mail: {\tt rossim@dima.unige.it}

\end{document}